\documentclass{amsart}
\usepackage{tikz} 
\usepackage{hyperref}
\usepackage{mathtools}
\usepackage{amssymb} 
\newtheorem{theorem}{Theorem}[section]
\newtheorem{lemma}[theorem]{Lemma}

\theoremstyle{definition}
\newtheorem{definition}[theorem]{Definition}

\newtheorem{cor}[theorem]{Corollary}
\theoremstyle{remark}
\newtheorem{proposition}[theorem]{Proposition}
\newtheorem{remark}[theorem]{Remark}
\newtheorem{problem}[theorem]{\bf Problem}
\numberwithin{equation}{section}

\begin{document}

\title[Nonexistence for complete K\"ahler Einstein  metrics]{Nonexistence for complete K\"ahler Einstein  metrics on some noncompact manifolds}

\author{Peng Gao}
\address{Department of mathematics, Harvard University, Cambridge, MA 02138, USA}
\email{penggao@math.harvard.edu}
\author{Shing-Tung Yau}
\address{Department of mathematics, Harvard University, Cambridge, MA 02138, USA}
\email{yau@math.harvard.edu}
\author{Wubin Zhou}
\address{Shanghai Center for Mathematical Sciences,
Shanghai  200433, China}
\email{wubin\_zhou@fudan.edu.cn}


\date{March 22, 2016}


\begin{abstract}
Let $M$ be a compact K\"ahler manifold and $N$ be a subvariety with codimension greater than or equal to 2. We show that there are no complete K\"ahler--Einstein metrics on $M-N$.  
As an application, let $E$ be an exceptional divisor of $M$.  Then  $M-E$ cannot admit any complete K\"ahler--Einstein metric if blow-down of $E$ is a complex variety with only canonical or terminal singularities. A similar result is shown for pairs. 
\end{abstract}

\maketitle

\section{Introduction and Main Theorem}
A basic question in K\"ahler geometry is how to find on each K\"ahler manifold a canonical metric such as  K\"ahler--Einstein  metric,  constant scalar curvature K\"ahler  metric,  or even  extremal metric.  When the K\"ahler manifold is compact with negative or zero first Chern class, the question has been solved by the senior author's celebrated work on Calabi's conjecture \cite{Yau3}. Whereas the first Chern class is positive,  there is an obstruction called Futaki invariant for the existence of K\"ahler--Einstein metric.

 For the noncompact case,  The majority of work focuses on open K\"aher manifolds or quasi-projective manifolds.  A complex manifold $M$ is open (resp. quasi-projective) if there is a compact K\"ahler (resp. projective) manifold $\bar M$ with an effective divisor $D$ such that $M$ is biholomorphic to $\bar M-D$.  
 The second author raised the question concerning existence of complete K\"ahler-Einstein metrics on quasi-projective varieties in \cite{Yau4}, where the results of \cite{CY} were announced for the case with constant negative scalar curvature. Also, the second author announced in \cite{Yau4} the existence of complete Ricci-flat K\"ahler metrics on the complement of an anticanonical divisor, using methods following his earlier work in \cite{Yau3,Yau}.
 
 Cheng and Yau \cite{CY} constructed complete negative K\"ahler-Einstein metrics on  $\bar M-D$ when $D$ is a normal crossing divisor and $K_M+D$ positive.
  In \cite{TY1,TY2} the second author showed that there exists a complete Ricci flat K\"ahler metric on quasi-projective  $M$, if $D$ is a neat and almost ample smooth divisor on $M$; or $\bar M$ is a compact K\"ahler orbifold and $D$ is a neat, almost ample and admissible divisor on $\bar M$. This follows the analysis of \cite{Yau3}. In \cite{TY3} he also proved on open manifold $M$ there are complete K\"ahler-Einstein metrics with negative scalar curvature if the adjoint canonical bundle $K_M+D$ is ample. This too follows the analysis of \cite{Yau3}.

Several years ago, the second author \cite{FYZ2} proposed the following questions
\begin{problem}	Let $M$ be a compact K\"ahler manifold and $N$ be a subvariety with codimension bigger than or equal to 2, how to find a complete canonical metric on the noncompact K\"ahler manifold $M-N$?
\end{problem}

In the Ricci-flat case, this was answered in \cite{Yau4} based on a theorem proved in \cite{Yau} for volume of complete noncompact Riemannian manifolds.
In order to handle this question, the authors (in \cite{FYZ2,FYZ1}) introduced the concept of complete metrics with  Poincar\'e-Mok-Yau (PMY) asymptotic property, and  constructed many constant scalar curvature K\"ahler metrics with  PMY asymptotic property on  some special types of noncompact K\"ahler manifolds. Since  these PMY type metrics are not K\"ahler--Einstein,  naturally one can ask  the following question
\begin{problem}\label{p2}
Can $M-N$  be endowed with complete K\"ahler--Einstein  metrics?
\end{problem}

In fact, little is known about the obstruction for the existence of complete K\"ahler--Einstein on noncompact K\"ahler manifolds.  The unique outstanding result is due to Mok and Yau's main theorem in \cite{MY} which states that a bounded domain $\Omega$ admits a complete K\"ahler--Einstein  if and only if $\Omega$ is a domian of holomorphy. If $N$ is a higher codimenion subvariety of a bounded domain $\Omega$, then $\Omega-N$ is not a domain of holomorphy. This implies on $\Omega-N$ there are no complete K\"ahler metrics with nonpositve Ricci curvature.   In our discussions on Problem \ref{p2}, the senior author proposed that the answer should be negative and the original work in \cite{MY} will give hints. Here we follow this idea and it turns out that we can obtian the following theorem
\begin{theorem}\label{thm1}
There are no complete K\"ahler metrics $\omega$ on $M-N$ satisfying $-\lambda\leq Ric_\omega 
\leq 0$ with nonnegative constant $\lambda$. 
\end{theorem}
Since $M-N$ is noncompact, according to Bonnet-Myer's compactness theorem $M-N$ cannot  admit K\"ahler--Einstein metrics with positive Ricci curvature. Immediately, we find a negative answer for Problem \ref{p2} that 
\begin{cor}\label{cor1}
	There are no complete K\"ahler--Einstein metrics $\omega$ on $M-N$.
\end{cor}

	If $N$ was allowed to have singularities,  Corollary \ref{cor1}  would imply there are no  complete K\"ahler--Einstein metric on  $\bar M-D$, where $\bar M$ is a compact K\"ahler manifold that desingularizes $M$ and $D$ the exceptional locus of $\bar M$.  By Hironaka's theorem on resolution of singularities, such desingularizations always exist over a field of characteristic $0$, also true for analytical varieties. 
	More precisely, blowing up  $M$ along $N$, one obtains a new compact  K\"ahler manifold $\bar M=Bl_N(M)$. 
	Then $M-N$ is bi-holomorphic to $Bl_N(M)-D$ where $D$ is the exceptional set.  And it's clear  $\bar M-D$ has a complete K\"ahler-Einstein metric if and only if $M-N$ has one.
	
	Similarly one could ask about the inverse operation, whether there is a  complete K\"ahler Einstein metric on  $ M-E$ if $E$ is an exceptional divisor  which blows down to a singularity. It is natural to expect that the answer depends on the type of singularities produced for example by a divisorial contraction. In fact, by a local analysis we can show the following 
	\begin{theorem}\label{ths}
	Let $f: M\rightarrow X$ be two  complex normal varieties which are birational,  and $E$ the sum of the exceptional divisors. Then if $X$ has only canonical or terminal singularities, and the codimension two singular locus is non-empty in $X$, there are no complete K\"ahler-Einstein metrics on $M-E$.
\end{theorem}

This result holds in general dimension. But it is difficult to relax the assumption on codimension of the singular locus in $X$.  
Especially, in dimension two we obtain the following corollary.
\begin{cor}
If $M$ is a complex surface  and  the singularities of $X$ are of type A-D-E, then there are no complete K\"ahler--Einstein metrics on $M-E$.  
\end{cor}

We can generalize Theorem \ref{ths}  to the case of pairs $(M,D)$ with boundary divisor $D$. As our analysis essentially depends only on the curvature condition, it can be carried out for these open cases as well.  For pairs, we have
\begin{theorem}\label{thsp}
	If  the pair $(X,D)$ is {\rm klt}, then there are no complete K\"ahler-Einstein metrics on $M-E$, where $E$ is the exceptional divisor of a log resolution $f:M\rightarrow X$ where $E\cap f^{-1}({\rm supp}(D))$ is simple normal crossing. 
\end{theorem}	
It is clear any open manifold (non-complete variety) birational to these pairs also cannot admit K\"ahler-Einstein metrics.

In Theorem \ref{ths} and Theorem \ref{thsp} the adjoint canonical bundle $K_M+E$  is generally not ample. It demonstrates that the ampleness condition in \cite{TY1,TY2, TY3} and \cite{Wu}  is necessary.
Besides, various generalizations of K\"ahler--Einstein metrics to the singular setting for pairs have been proposed in the literature, see e.g.  \cite{EGZ09,BEGZ10,BG} etc. Our theorems assert immediately that these singular metrics  for klt pairs are not complete,  and  it has been shown for klt pairs these metrics have cone singularities \cite{Gu13}.

The rest of this note is devoted to  the proof of Theorem \ref{thm1}, Theorem \ref{ths} and Theorem \ref{thsp}.

\section{The Proof of  Theorem \ref{thm1}}

\begin{proof}[The proof of Theorem \ref{thm1}]
 We prove Theorem \ref{thm1} by contradiction and the key point is to show the boundedness of volume function nearby $N$.  Let $\dim_\mathbb{C}M=n$ and $\dim_{\mathbb C}N=k$ and $k\leq n-2$.  We assume that over $M- N$ there is a complete K\"ahler metric $\omega$ with Ricci curvature $$-\lambda\leq Ric_\omega\leq  0$$ with a nonnegative constant $\lambda$.  Locally,  let $p\in N$,  there is a local open ball  $ B_{\xi}=\{(z_1,z_2,z_3,\cdots,z_n),|z_i|<\xi\}$  such that $p=\{0,\cdots,0\}$ and $N\cap B_\xi$ is  $$z_{1}=z_{2}=\cdots=z_{n-k}=0.$$  
On this open set $B_\xi-N$, we assume that  $\omega=\sqrt{-1}g_{i\bar j}dz_i\wedge d\bar z_{j}.  $
Choose a two dimension polydisc $\Delta_p(\xi)$ to be $\big\{(z_1,z_2,0,\cdots,0)\big||z_1|\leq\xi, |z_2|\leq \xi\big\},$ then the exponent of the volume form of the metric, i.e. $\det g_{i\bar{j}}=e^{\log{\det(g_{i\bar{j}}) }}$ is plurisubharmonic over $\Delta_p(\xi)-p$  according to 
$$-Ric= \partial\bar\partial (\log \det(g_{i\bar j}))\geq 0.$$
\begin{center}
\begin{tikzpicture}[scale=1.3,>=latex]
\draw[](-1/2,2)..controls(0,1)..(0,0);
\draw[dashed](0,0)..controls(0,-1/2)..(1/2,-1.5);
	\node at (-1/4,0){$p$};
	\node at (-0,2){$N$};
	\node at (-2.5,0){$\Delta_p(\xi)$};
	\fill (0,0) circle (1.5pt);
\draw[] (-2,-1)..controls(0,-1)..(1.5,-1);
\draw[] (-1.5,1)..controls(-1/4,1)..(-1/5,1);
\draw[dotted] (-1/5,1)--(1/16,1);
\draw[] (1/16,1)--(2,1);
\draw[](-2,-1)--(-1.5,1);
\draw[](1.5,-1)--(2,1);
\draw[dashed,<->](1.2,-1)--(1.7,1);
\node at (1.2,0.4){$D_{z_2}$};
\draw[dotted,-](0,0)--(1.75,0); 
\node at (0.8,-1/4){$z_1$};
\node at (1.75,-0.75){$z_2$};
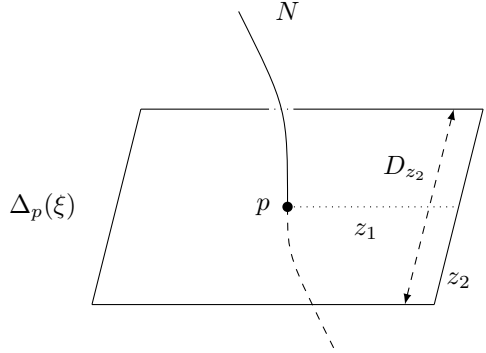
\node at (0,-2){the diagram of $\Delta_p(\xi)$};
\end{tikzpicture}
\end{center}  

Let $D_{z_2}$ be disc
$$D_{z_2}=\{ (z_1,z_2,0 \cdots,0)\  \text{with} \ |z_1| \text{ is fixed} \}\subset \Delta_p(\xi)-p,$$ 
by the maximum principle on the disc $D_{z_2}$ 
$$\det(g_{i\bar{j}})\big |_{D_{z_2}}\leq \sup_{|z_2|=\xi} \det(g_{i\bar{j}}).$$
Since the  set $\{(z_1,z_2,0,\cdots 0)\ \text{with}\  |z_2|=\xi\}$  in  $\overline{\Delta_p(\xi)}$    is a close subset of $M-N$,  we have
\begin{equation}\label{cdet1}
\det(g_{i\bar j}(z_1,z_2))\leq \sup_{|z_2|=\xi} \det(g_{i\bar{j}})\leq C_1\ \ \text{for}\ 0<|z_1|\leq\xi, |z_2|<\xi,
\end{equation}
with constant $C_1$. Similarly in  the same way, the following inequality is satisfied for some constant $C_2$ 
\begin{equation}\label{cdet2}
\det(g_{i\bar j}(z_1,z_2))\leq \sup_{|z_1|=\xi} \det(g_{i\bar{j}})\leq C_2\ \ \text{for}\ |z_1|\leq\xi, 0< |z_2|<\xi.
\end{equation}
Combining  (\ref{cdet1}) and (\ref{cdet2}),  one has  
\begin{equation}\label{cdet}
\det(g_{i\bar j})\big|_{\Delta_p(\xi)-p}\leq \sup_{|z_1|=\xi, |z_2|=\xi} \det(g_{i\bar{j}})=\max\{C_1, C_2\}.
\end{equation}

On the other hand,  we \textbf{claim}
 \begin{equation}\label{claim}
\omega>C_3 dz_i\wedge d\bar z_i
\end{equation}  on the  ball $B_{\frac{\xi}{2}}-N $ with some constant $C_3$ and  $B_{\frac{\xi}{2}}= \{(z_1,\cdots,z_{n}), |z_i|<\frac{\xi}{2}\}$,  which yields
\begin{equation}\label{below}
	\det(g_{i\bar{j}})>C_4
\end{equation}
for some constant $C_4$.   In fact,  put  Poincar\'e metric $\omega_P$ on $B_{\xi}$, and let $u=\text{trace}_\omega \omega_P$. It follows that
\begin{equation}
\omega_P\leq u \ \omega.
\end{equation}
Applying Chern-Lu inequality\cite{Yau2}, we have
\begin{equation}\label{Ch-Lu}
\Delta u\geq -\lambda u+c u^2
\end{equation}
where $-c$ is the upper bound of the  bisectional curvature of $\omega_P$.    Since the metric $\omega$ is not complete on the boundary of $B_{\frac{\xi}{2}}$,  Yau's Schwarz lemma  \cite{Yau2} can not applied directly  for identity map $Id: (B_{\frac{\xi}{2}}-N,\omega)\to (B_\xi,\omega_P)$ to obtain the upper bound of $u$.
But  we can choose a sequence $p_k\in B_{\frac{\xi}{2}}$, such that $u(p_k) \to \sup_{B_{\frac{\xi}{2}}-N} u$. There are two cases for the limit points of $p_k$.  If $p_k\to p_0\in \bar B_{\frac{\xi}{2}}-N$, one has $\sup_{B_{\frac{\xi}{2}}} u\leq \frac{\lambda}{c} $ according to maximum principle for inequality (\ref{Ch-Lu}).  Otherwise, the $p_k\to p_0\in N$,  it is the same as  \cite{Yau2}  that we can obtain the boundedness of $u$ according to Yau's general maximum principle  since $\omega$ is complete nearby $N$.  
 In both cases, $u$ is bounded and claim (\ref{claim}) is obtained.

Using (\ref{cdet})  and (\ref{below}), we have 
$C_4<\det(g_{i\bar j})\big |_{\Delta_p(\frac{\xi}{2})-p}<C_1$. 
	It follows $$g_{1\bar 1}\big |_{\Delta_p(\xi)-p}\leq C$$
 for some constant $C$.
Let $z_1=x_1+\sqrt{-1} y_1$, the length of the real line segment in space $\{(z_1,0,\cdots,0)\}$ from $(\xi,0,\cdots,0)$ to  $p$ is given now by 
\begin{equation}
	\int_0^\xi\sqrt{g_{1\bar 1}}d x_1\leq \sqrt{C}\xi
\end{equation}
which contradicts the completeness of the metric $\omega$ nearby $N$, and we have finished the proof.
\end{proof}
\begin{remark}
From the proof above, we see that all analysis is local (that is,  there are no K\"ahler metric $
\omega$ on $B_\xi-N$ such that $\omega$ is complete nearby $N$).  It means Theorem \ref{thm1} is still true when  $M$ is  not compact which therefore allows us to prove similar statements in cases when more than one smooth subvariety of codimension $2$ or higher are removed. 
\end{remark}

\section{The Proof of  Theorem \ref{ths}}

Although our proof shall be local and mainly differential geometric in nature, it is useful to phrase the theorem in the right algebraic setting. The notations are standard in the literature, the readers can consult the standard \cite{KM98} for further detail. 

\subsection{Canonical singularities}
First we recall some properties of the class of singularities known as {\it canonical singularity}. We do not restrict to any specific number of dimension of the variety. 

Unless otherwise stated, we will be considering normal varieties. Then by Zariski's main theorem which applies to normal varieties, we know that the fundamental locus of all birational morphisms are closed subvarieties of codimension at least $2$. 

Denote by $M_{\rm reg}=M-M_{\rm sing}$ the complement of the singular set $M_{\rm sing}$, we then have the open immersion $M_{\rm reg}\xhookrightarrow{\it i} M$. The canonical sheaf $\omega_M=i_*\omega_{M_{\rm reg}}$ exists by Hartog's extension theorem and $M$ being normal. We note that the fundamental locus is generally not contained in $M_{\rm sing}$, and vice versa. The former clearly depends on the choice of a birational morphism. 

\begin{definition}[Canonical singularity]\label{can}
A $n$-dimensional normal variety $M$ with $\mathbb{Q}$-Cartier canonical divisor $K_M$ is said to have only canonical singularities, if there exists a birational morphism $f: Y\rightarrow M$ from a smooth variety $Y$ such that in the ramification formula 
$$
K_Y=f^*K_M+\sum a_i E_i
$$
$a_i\ge 0$ for all divisors $E_i$ which are exceptional.
\end{definition}

We shall need later the following result (3.4)(A) from \cite{Reid87}. 
\begin{theorem}[Canonical $\Rightarrow$ Du Val in codim $2$]\label{canduval}
Let $M$ be a $n$-fold with canonical singularities, not necessarily isolated, then $M$ is isomorphic to the following form analytically
$$
M \cong (\mbox{Du Val sing.})\times \mathbb{A}^{n-2}
$$
 in the neighborhood of a general point of any codimension $2$ stratum.
\end{theorem}
Where surface canonical singularities are Du Val singularities, also called A-D-E singularities. This can be easily proved using the `general section theorem' for canonical singularities, see e.g. theorem (1.13) in \cite{C3-f}.

\subsection{The Proof of  Theorem \ref{ths}}

To proceed we need also the following lemma which proves a version of theorem \ref{ths} for a canonical (Du Val) singularity in $2$ dimensions.

\begin{lemma}[Du Val surface singularity]\label{dvlemma}
If $M$ is a normal complex surface with only canonical singularities, then there does not  exist a complete K\"ahler-Einstein metric on the complement of a finite number of points on $M$.
\end{lemma}

\begin{proof}[The proof of  the Lemma]
Here we will show the case that there is only one  exceptional divisor $E$ and  $\bar{M}$ is smooth. If $f(E)$ consists purely of  smooth points of $M$, we get form Theorem \ref{thm1} that there cannot exist complete K\"ahler metrics on  $M-f(E)$ with Ricci curvature satisfying $-\lambda\leq Ric\leq 0$. So we assume  $f(E) \subset M_{\rm sing}$.  

Recall a Du Val singularity is given locally by a polynomial in $\mathbb{A}^3$, i.e. a hypersurface. By Weierstrass preparation theorem, near $p$, $M-p$ can be covered by finitely many disjoint open sets   $B_i^* $ such that each $B_i^*$ is biholomorphic to standard punctured disc $B^*=\{(z_1,\cdots,z_n), 0<\sum_{i=1}^{n}|z|^2<1\}$, $n=2$. Then if there is a complete K\"ahler--Einstein metric $\omega$ on  $M-p$, each $B_i^*$ admits a K\"ahler--Einstein metric $\omega\big|_{B_i^*}$. It implies  $B^*$ can be endowed with a K\"ahler--Einstein metric that is complete at the punctured point. But this is impossible from the proof of Theorem \ref{thm1}.
This finishes the proof.
\end{proof}

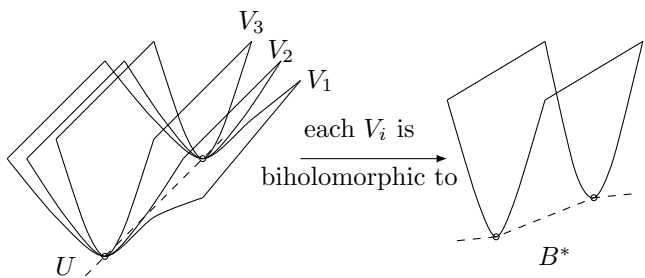
\begin{center}
\begin{tikzpicture}[scale=1.3,>=latex]
\draw[dashed](-0.2,-0.2)--(0,0);
\draw[dashed](0,0)--(1.2,1.2);
\draw plot[smooth] coordinates {(-0.5,1.2) (-0.2,0.3) (0,0)  (0.2,0.3) (0.5,1.2) };
\draw plot[smooth] coordinates {(-0.8,1) (-0.4,0.4) (0,0) (0.4, 0.4)  (0.8,1) };
\draw plot[smooth] coordinates {(-0.5+1,1.2+1) (-0.2+1,0.3+1) (1,1)  (0.2+1,0.3+1) (0.5+1,1.2+1) };
\draw plot[smooth] coordinates {(-0.8+1,1+1) (-0.4+1,0.4+1) (0+1,0+1) (0.4+1, 0.4+1)  (0.8+1,1+1) };
\draw plot[smooth] coordinates {(-1,1) (-0.5,0.4) (0,0)  (0.5,0.4) (1,0.6) };
\draw plot[smooth] coordinates {(-1+1,0.8+1.2) (-0.5+1,0.4+1) (0+1,0+1)  (0.5+1,0.4+1) (1+1,0.8+1) };
\draw[-](-0.5,1.2)--(0.5,2.2);
\draw[-](-0.8,1)--(0.2,2);
\draw[-](0.5,1.2)--(0.5+1,1.2+1);
\draw[-](0.8,1)--(0.8+1,1+1);
\draw[-](-1,1)--(-1+1,0.8+1.2);
\draw[-](1,0.6)--(1+1,0.8+1);
\node at (0.5+1,1.2+1.2){$V_3$};
\node at (0.5+1.3,1.2+0.9){$V_2$};
\node at (0.5+1.7,1.2+0.6){$V_1$};
\node at (-0.4, -0.1){$U$};
\draw[] (0, 0) circle (0.03);
\draw[] (1, 1) circle (0.03);
\draw[->](2,1)--(3.5,1);
\node at (2.6,1.3){each $V_i$ is };
\node at (2.6,0.8){biholomorphic to};
\draw plot[smooth] coordinates {(-0.5+4,1+0.6)(-0.2+4,0.6) (4,0+0.2)(0.2+4,0.6) (0.5+4,1+0.6) };
\draw plot[smooth] coordinates {(-0.5+5,1+0.6+0.6) (-0.2+5,0.3+0.6) (5,0.6)  (0.2+5,1) (0.5+5,1.2+1) };
\draw[-] (-0.5+4,1+0.6)--(-0.5+5,1+0.6+0.6);
\draw[-](0.5+4,1+0.6)--(0.5+5,1.2+1);
\draw[dashed](3.6,0.16)--(4,0+0.2)--(5,0.6)--(5.4,0.64);
\node at (4.6,0){$B^*$};
\draw[] (4, 0.2) circle (0.03);
\draw[] (5, 0.6) circle (0.03);

\end{tikzpicture}
\end{center}  

 Du Val singularities are also quotients of $\mathbb{C}^2$ by finite subgroups of $SL(2,\mathbb{C})$. By Cartan's lemma any quotient singularity is isomorphic to $\mathbb{C}^n/G$ with $G\subset GL(n,\mathbb{C})$ a finite group. The following theorem applies to quotient singularities in all dimensions $n\ge 2$ by small subgroups $G$ of $GL(n,\mathbb{C})$. They are fixed point free in codimension $1$. 

\begin{proposition}[Quotient singularities]\label{quotsing}
Let $G\subset GL(n,\mathbb{C})$ be a small finite group acting on $\mathbb{C}^n$, then there does not exist a complete K\"ahler-Einstein metric on the complement of the singular locus in $\mathbb{C}^n/G$.
\end{proposition}

\begin{proof}
 Let  $S=\{x\in \mathbb{C}^n | x=g(x) \mbox{ for some } g\ne 1\}$. It is standard that the singular locus is $\widetilde S=S/G$. Let $V^*=B-\widetilde S$ with $B=\big\{(z_1,\cdots,z_n), |z_i|^2<1\big\}$  be an open neighborhood of the singular locus. 
 For small enough $V^*$, it is covered by a disjoint union of open sets $V_i$ $(i=1,\cdots, l)$ for some finite integer $l\le |G|$.
 If on $\mathbb{C}^n/G-\widetilde S$ there is a complete K\"ahler--Einstein metric $\omega$, then there is a $G$-equivariant  K\"ahler--Einstein metric $\omega^G$, whose restriction $\omega^G|_{V_i}$ to each $V_i$ is a K\"ahler--Einstein metric which is complete at $S\cap V_i$.  From the proof of Theorem \ref{thm1} this is impossible. 
\end{proof}

\begin{proof}[The proof of Theorem \ref{ths}]
here is the main's proof. 
We make use of Theorem \ref{canduval}, by which the claim reduces to showing non-existence of K\"ahler-Einstein metrics for the complement of the singular points which are Du Val on the codimension two strata. This stratum is non-empty by assumption and the statement then follows from Lemma \ref{dvlemma}. 
\end{proof}

\section{The proof of Theorem \ref{thsp}}

\subsection{Log Terminal Singularities}\label{lt}

For simplicity, we will only consider normal varieties. Then by Zariski's main theorem, we know that the fundamental locus of all birational morphisms involved in our setup are closed subvarieties of codimension at least $2$. Denote by $M_{\rm reg}=M-M_{\rm sing}$ the complement of the singular set $M_{\rm sing}$, we then have the open immersion $M_{\rm reg}\xhookrightarrow{\it i} M$. The canonical sheaf $\omega_M=i_*\omega_{M_{\rm reg}}$ exists by Hartog's extension theorem and $M$ being normal. We note that the fundamental locus is generally not contained within $M_{\rm sing}$, and vice versa.

To measure the local valuative property of a singularity, the notion of discrepancy was introduced (see e.g. \cite{Kol96}). This notion can be defined for $\mathbb{R}$-linear combinations of Weil divisors  $D=\sum a_i D_i$ but often $\mathbb{Q}$-linear is sufficient. Since Weil divisors are not always pulled back, we assume $K_M+D$ is $\mathbb{Q}$-Cartier. 
Let $f:Y\rightarrow M$ be a birational (bimeromorphic) morphism, and $Y$ normal, we let
$$
K_Y +f^{-1}{}_*(D) = f^*(K_M+D)+\sum a_{E}(M,D) E
$$
where $E$ are distinct prime divisors of $Y$ and $a_{E}(M,D)\in \mathbb{Q}$. This is now a statement regarding linear equivalence of divisors and not local. 
As non-exceptional divisors appear in this formal sum, the right hand side needs some explanation. 

We demand as in  \cite{Kol96} that for a non-exceptional divisor $E$,  the coefficient $a_{E}(M,D) \ne 0$ iff $E=(f^{-1})_*(D_i)$ for some $i$, in which case we set $a_{D_i}(M,D) =-a_i$. This defines $a_{E}(M,D)\in \mathbb{R}$, called the {\it discrepancy} of $E$ with respect to the pair $(M,D)$.

A more global measure of singularity of the pair is defined by taking the ${ \inf}$ of  $a_{E}(M,D)$ over distinct primes $E\subset Y$. 
\begin{definition}\label{discrp}
\begin{eqnarray}
{\rm discrp}(M,D)&:=&\inf_E\{a_{E}(M,D)| \mbox{ E$\subset$ Y exceptional with center $\ne \varnothing$ on M}\}\nonumber\\
{\rm discrp}_{\rm total}(M,D)&:=&\inf_E\{a_{E}(M,D)| \mbox{ E$\subset$ Y has non-empty center on M}\}\nonumber
\end{eqnarray}
\end{definition}

We are ready to recall the definition of a log terminal `singularity'. The following is  def. $(2.34)$ in \cite{KM98}
\begin{definition}\label{klt}
Let $(M,D)$ be as above, then we say $(M,D)$ or $K_M+D$ is 
\begin{eqnarray}
\left.\begin{array}{c}
terminal\\
canonical\\
klt\\
lc\\
\end{array}\right\}
\quad\mbox{if discrp(M,D)}
\left\{\begin{array}{c}
$$> 0$$\\
$$\ge 0$$\\
$$> -1$$ \,\,\mbox{and}\,\, $$ \lfloor D \rfloor \le 0$$\\
$$\ge -1$$
\end{array}\right.
\end{eqnarray}
\end{definition}
The running of MMP preserves both the Kawamata log terminal (klt) and log canonical (lc) property of pairs. This means the statements we make may easily be extended to the context of birational geometry. 

\subsection{Local property of klt pairs}

We will need the following lemma characterizing klt singularity structures in codimension $2$. This comes from a cutting by hypersurface technique similar to the case of complete varieties, and for details we refer to \cite{GKKP}.

\begin{lemma}[Property in codimension $2$]\label{kltquo}
Let $(X,D)$ be a klt pair. Then there exists a closed subset $N\subset X$ with ${\rm codim}_X N\ge 3$ such that $X\backslash N$ has quotient singularities.
\end{lemma}

This is proposition (9.3) in \cite{GKKP}. However, using \cite{KM98}, Corollary (2.39), this will follow as a corollary of Theorem \ref{canduval}. In that case, we can give a proof of  \ref{thsp} based on \ref{ths}, since we may consider the pair $(X,\varnothing)$. However, we will end up with a redundant restriction on the codimension of the singular locus. 
The proof of Theorem \ref{thsp} now parallels that of \ref{ths} in the previous section. 
\begin{proof}[The proof of Theorem \ref{thsp}]
By resolution of singularities the log resolution claimed in the theorem exists. Let $S= f(E)\cap X_{\rm sing}$. The existence of a complete K\"ahler-Einstein metric on the complement of the exceptional divisor $M\backslash E$ is equivalent to the exisence of a complete K\"ahler-Einstein metric on $X\backslash S$.
By lemma \ref{kltquo}, we only have to show the case when ${\rm codim}_X S= 2$, in which case each irreducible component of $S$ is a finite quotient.
Clearly $S$ has only finitely many irreducible components, and the non-existence of  a complete K\"ahler-Einstein metric on  $X\backslash S$ follows from Prop. \ref{quotsing}.
\end{proof}

A simple application of our results is in the context of desingularization of the Satake compactifications\cite{Ig}. Without using any detail on modular forms, we can conclude that the compactification of locally symmetric spaces must proceed by adding boundary components which contain singular points, i.e. cusps. This follows from the simple fact that arithmetic quotients admit complete K\"ahler-Einstein metrics.

\end{document}